\documentclass[a4paper,11pt,reqno]{amsart}
\usepackage[T1]{fontenc}
\usepackage[utf8]{inputenc}
\usepackage{lmodern}
\usepackage{amsmath, amsthm, amssymb,amscd, mathrsfs, amsfonts, mathtools}
\usepackage{hyperref}
\usepackage{euler}

\usepackage[all]{xy}
\usepackage{todonotes}
\usepackage{xcolor}

\usepackage[lite]{amsrefs}

\newcommand*{\arxiv}[1]{\href{http://www.arxiv.org/abs/#1}{arXiv: #1}}
\renewcommand{\PrintDOI}[1]{\href{http://dx.doi.org/\detokenize{#1}}{doi: \detokenize{#1}}%
  \IfEmptyBibField{pages}{, (to appear in print)}{}}

\def\commutatif{\ar@{}[rd]|{\circlearrowleft}}

\newcommand{\eq}[1][r]
   {\ar@<-3pt>@{-}[#1]
    \ar@<-1pt>@{}[#1]|<{}="gauche"
    \ar@<+0pt>@{}[#1]|-{}="milieu"
    \ar@<+1pt>@{}[#1]|>{}="droite"
    \ar@/^2pt/@{-}"gauche";"milieu"
    \ar@/_2pt/@{-}"milieu";"droite"}

\def\dar[#1]{\ar@<2pt>[#1]\ar@<-2pt>[#1]}
  \entrymodifiers={!!<0pt,0.7ex>+} 

\newcommand{\bigon}[4][r]{
    \ar@/^1pc/[#1]^{#2}_*=<0.3pt>{}="HAUT"
    \ar@/_1pc/[#1]_{#3}^*=<0.3pt>{}="BAS"
    \ar@{=>} "HAUT";"BAS" ^{#4}
  }

\newcommand{\bigons}[6][r]{  
    \ar@/^2pc/[#1]^{#2}_*=<0.3pt>{}="HAUT"
    \ar@{}    [#1]     ^*=<0.3pt>{}="MILIEUHAUT"
                       _*=<0.3pt>{}="MILIEUBAS"
    \ar[#1]_(0.3){#3}                  
    \ar@/_2pc/[#1]_{#4}^*=<0.3pt>{}="BAS"
    \ar@{=>} "HAUT";"MILIEUHAUT" ^{#5}
    \ar@{=>} "MILIEUBAS";"BAS" ^{#6}
  }

\newtheorem{thm}{Theorem}[section]
\newtheorem{pro}[thm]{Proposition}
\newtheorem{lem}[thm]{Lemma}
\newtheorem{cor}[thm]{Corollary}

\theoremstyle{definition}
\newtheorem{df}[thm]{Definition}

\theoremstyle{remark}
\newtheorem{rmk}[thm]{Remark}

\newtheorem{ex}[thm]{Example}

\allowdisplaybreaks

\newcommand\id{1}
\newcommand{\Id}{\operatorname{id}}

\newcommand{\R}{\mathbb{R}}

\newcommand\rTo{\longrightarrow}

\newcommand\mto{\longmapsto}

\newcommand\rRack{\triangleleft}

\let\fr\mathfrak
\let\cal\mathcal

\let\bb\mathbb

\def\a{\alpha}
\def\b{\beta}
\def\g{\gamma}
\def\d{\delta}

\def\s{\sigma}
\def\vp{\varphi}

\hypersetup{
  colorlinks = true,
  urlcolor = blue,
  linkcolor = blue,
  citecolor = red,
  pdfauthor = {Elhamdadi, M., Moutuou, E.},
  pdfkeywords = {Finitely stable racks, Quandles, extensions of quandles},
  pdftitle = {Elhamdadi, M., Moutuou, E. - Finitely stable racks and rack representation},
  pdfsubject = {racks, representations, quandles},
  pdfpagemode = UseNone
}

\title{Finitely stable racks and rack representations}

\author{Mohamed Elhamdadi} 
\address{Department of Mathematics, 
University of South Florida, Tampa, FL 33620, U.S.A.} 
\email{emohamed@math.usf.edu} 

\author{El-ka\"ioum M. Moutuou} 
\address{Department of Mathematics, 
University of South Florida, Tampa, FL 33620, U.S.A.} 
\email{elkaioum@moutuou.net}

\begin{document}

\maketitle

\begin{abstract}
We define a new class of racks, called finitely stable racks, which, to some extent, share various flavors with Abelian groups. Characterization of finitely stable Alexander quandles is established. Further, we study twisted rack dynamical systems, construct their cross-products, and introduce representation theory of racks and quandles. We prove several results on the {\em strong} representations of finite connected involutive racks analogous to the properties of finite Abelian groups. Finally, we define the {\em Pontryagin} dual of a rack as an Abelian group which, in the finite involutive connected case, coincides with the set of its strong irreducible representations.
\end{abstract}

\tableofcontents

\section{Introduction}

{\em Racks and quandles} were introduced independently in the 1980s by Joyce and Matveev~\cites{Joyce:Thesis, Matveev} mainly in order to provide algebraic machineries in constructing invariants of knots and links. Racks are sets equipped with non-associative structures which, to some extent, share much with the theory of groups, Hopf algebras, Lie algebras, etc. More precisely, a {\em rack} is a set $X$ equipped  with a binary operation $$X\times X\ni (x,y)\mto x\rRack y \in X$$ which is bijective with respect to the left variable and satisfying 
\[
(x\rRack y)\rRack z=(x\rRack z)\rRack (y\rRack z)
\]
for all $x,y,z\in X$. If in addition $x\rRack x= x$ for each $x\in X$, then $(X,\rRack)$ is called a {\em quandle}. A rack or a quandle $X$ is {\em trivial} if $x\rRack y=x$ for all $x,y\in X$. Unless otherwise specified, all our racks and quandles will be assumed nontrivial. The most natural quandles are the {\em conjugation} and the {\em core} quandles of a group. Specifically, let $G$ be a group. Its conjugation quandle $Conj(G)$ is the set $G$ together with the operation $g\rRack h:=hgh^{-1}$, while its {\em core quandle} $Core(G)$ is the set $G$ equipped with the operation $g\rRack h:=hg^{-1}h$. We refer to \cite{Joyce:Thesis, Matveev,Elhamdadi-Nelson:Quandles_An_Introduction} for more details about racks and quandles.

In this paper we treat racks and quandles purely as algebraic objects on their own right rather than by their connections with knot theory. We introduce and investigate several new notions and algebraic properties in the category of racks. In~\cite{Elhamdadi-Moutuou:Foundations_Topological_Racks} we defined a {\em stabilizer} in a rack $X$ to be an element $u\in X$ such that $x\rRack u = x$ for all $x\in X$. For example, given a group $G$, the stabilizers of Conj$(G)$ are exactly the elements of the center $Z(G)$ of $G$. On the other hand, the associated {\em core quandle} of $G$ has no stabilizers if $G$ has no non-trivial $2$--torsions. This observation suggests that the property of having stabilizers is too strong to capture the identity and center of a group in the category of racks and quandles. We have weakened this property by introducing the notion of {\em stabilizing families} in a rack. A finite subset $\{u_1, \ldots, u_n\}$ of a rack $X$ is a {\em stabilizing family of order $n$}  for $X$ if $(\cdots (x\rRack u_1)\rRack \cdots )\rRack u_n = x$, for all $x\in X$. If such a family exists, $X$ is said to be {\em finitely stable} or $n$--{\em stable}. The (possibly empty) set of all stabilizing families of order $n$ is denoted by $\cal S^n(X)$. Let, for instance, $\{x_1,\ldots, x_k\}$ be a finite subset of a group $G$. Then we get a stabilizing family of order $2k$ for $Core(G)$ by duplicating each element in the first set; {\em i.e.}, $\{x_1, x_1, x_2, x_2, \ldots, x_k, x_k\}$ is a stabilizing family of order $2k$. It follows that $Core(G)$ is $2k$--stable, for all $k\leq |G|$ (cf. Proposition~\ref{pro:Core-2k}). However, for the odd case, we have the following (cf.Theorem~\ref{thm:Core-odd}) result 

\medskip
\noindent {\em Let $G$ be a group. Then $Core(G)$ is $(2k+1)$--stable if and only if $G$ is isomorphic to a direct sum $\oplus_I\bb Z_2$ of the cyclic group of order $2$.}
\\

We have paid attention to the interesting class of {\em Alexander quandles} and established a general criterion for them to be finitely stable. Recall~\cites{Joyce:Thesis, And-Grana} that given a group $\Gamma$ and a $\bb Z[\Gamma]$--module $M$, each $\g\in \Gamma$ provides $M$ with the quandle structure $\rRack^\g$ defined by $x\rRack^\g y:= (x-y)\cdot \g + y, \ x,y\in M$. For all non-negative integer $n$ and each $\g\in \Gamma$, let $F_\g: M^n \rTo M$ be the function given by $F_\g(x_1,\ldots, x_n):= \sum_{i=1}^nx_i\cdot \g^{n-i}$. Then we have following result (Theorem~\ref{thm:Alexander})\\

\noindent {\em The quandle $(M, \rRack^\g)$ is $n$--stable if and only if $\g$ is an $n$--torsion. Furthermore, $\cal S^n(M)$ is the abelian group of all solutions of the equation $$F_\g(x_1,\ldots, x_n)=0.$$}

We have extended the notion of dynamical cocycles introduced in~\cite{And-Grana} to that of {\em twisted rack dynamical systems}, which are triples $(Q,X,\partial)$ where $X$ and $Q$ are racks, together with an action by rack automorphisms of $X$ on $Q$, and a family $\{\partial_{x,y}\}_{X \times X}$ of functions $Q\times Q\rTo Q$ satisfying some compatibily conditions with respect to the rack structure of $X$. When the maps $\partial_{x,y}$ happen to be rack structures on $Q$, we say that $(Q,\partial_{x,y})_X$ is an {\em $X$--bundle of racks}. In particular, for a given group $G$, $G$--{\em families of quandles} studied in~~\cites{Ishii:G-family_Quandles, Nosaka:Quandles_Cocycles} are special cases of bundles of racks. Associated to a twisted rack dynamical system $(Q,X,\partial)$ there is a rack $Q\rtimes_\partial X$, its {\em cross--product}, which is the Cartesian product $Q\times X$ equipped with the operation $(p,x)\rRack_\partial (q,y):=(\partial_{x,y}(p,q),x\rRack y)$. \\

Further, elements of rack representation theory are introduced and general properties are established. Specifically, we showed the analogues of the Schur's lemma for racks and quandles, and proved that (cf. Theorem~\ref{thm:strong_rep_connected}) \\

\noindent {\em Every strong irreducible representation of a finite connected involutive rack is one--dimensional.} \\

Finally, an analogous notion of the Pontryagin dual is defined for racks. The Pontryagin dual of a rack $X$, denoted by $D_qX$, is an Abelian group; in fact it is a finite copies of the unitary group $U(1)$. \\

The article is organized as follows. In section~\ref{fsr}, we introduce the concept of finitely stable racks which presents itself as an appropriate analogue of "Abelianity" in the category of racks and quandles. We show the existence of non-trivial stabilizing families for the core quandle $Core(G)$ and characterize stabilizing families for the $Cong_{\phi}$ quandle where $\phi$ is an automorphism of $G$. In section~\ref{SAQ}, we give necessary and sufficient conditions for Alexander quandles to be finitely stable, and provide a general algorithm to construct stabilizing families.  We introduce and study in Section~\ref{npivot} the $n$-pivot of a group as a generalization of the $n$-core defined by Joyce, {\em approximate units} in section ~\ref{rackaction}, twisted rack dynamical systems and their cross-products in Section~\ref{twist}, and $X$--bundles in section~\ref{Xbundle}. Section~\ref{rackrep} presents 
elements of representation theory of racks and quandles. Section~\ref{strong_rep} is mainly devoted to strong irreducible representations of finite connected involutive racks, and in Section~\ref{duality} we define characters of a rack and duality.


\section{Finitely stable racks}\label{fsr}

Throughout, we will use the following notations: if $u_1,..., u_n$ are elements in the rack $X$, then for $x\in X$ we will write 
\[
x\rRack (u_i)_{i=1}^n := (\cdots (x\rRack u_1)\rRack \cdots )\rRack u_n.
\] 
In particular, if $u_i=u$ for all $i=1,...,n$, then we will write $x\rRack^nu$ for $x\rRack(u)_{i=1}^n$.

Recall from~\cite{Elhamdadi-Moutuou:Foundations_Topological_Racks} that a \emph{stabilizer} in the rack $X$ is an element $u\in X$ such that $x\rRack u = x$ for all $x\in X$. This notion can be generalised as follow.

\begin{df}
Let $X$ be a rack or quandle.
\begin{enumerate}
\item A \emph{stabilizing family of order} $n$ for $X$ is a finite subset $\{u_i, i=1, \ldots, n\}$ of $X$ such that 
\[
x\rRack (u_i)_{i=1}^n = x,
\]
for all $x\in X$. In other words, $R_{u_n}R_{u_{n-1}}\cdots R_{u_1} = Id_X$.
\item An $n$--{\em stabilzer} of $X$ is a stabilizing family of order $n$ of the type $\{u\}_{i=1}^n$.
\item For any positive integer $n$, we define the \emph{$n$--center} of $X$, denoted by $\cal S^n(X)$, to be the (possibly empty) collection of all stabilizing families of order $n$ for $X$. 
\item The collection $$\cal S(X):= \bigcup_{n\in \bb N}\cal S^n(X)$$ of all stabilizing families for $X$ is called the \emph{center} of the rack $X$.
\end{enumerate}
\end{df}

\begin{lem}
Let $X$ be a rack and $\{u_i\}_{i=1}^n\in \cal S^n(X)$. Let $\s$ be an element in the symmetric group $\fr S_n$. Then $\{u_{\s(i)}\}_{i=1}^n\in \cal S^n(X)$ if $\s$ is an element of the cyclic group of order $n$ generated by the permutation $\left.\begin{pmatrix}
  2 & \ldots & n & 1
\end{pmatrix}
\right.
.$
\end{lem}

\begin{proof}
Straightforward.
\end{proof}

\begin{df}
A rack $X$ is said to be \emph{finitely stable} if $\cal S(X)$ is non-empty. It will be called $n$--\emph{stable} if it has a stabilizing family of order $n$.
\end{df}

\begin{rmk}
Notice that any stabilizer $u$ in $X$  is obviously a $1$--stabilizer. Moreover, if in particular every element of $X$ is an $n$--stabilzer, then we recover the definition of an $n$--{\em quandle} of Joyce~\cite{Joyce:Thesis} (see also~\cite{Hoste-Shanahan:n-Quandles}).
\end{rmk}

\begin{ex}
Any finite rack $X$ is finitely stable since the symmetric group of $X$ is finite.
\end{ex}

We will see later that a rack might have only stabilizing families of higher orders.

\begin{ex}
Consider the real line $\R$ with the usual rack structure $$x\rRack y = 2y - x, \ \ x, y\in \R.$$ 
Given any family $\{x_1, \ldots, x_n\}$ of real numbers we get the formula
\[
t\rRack(x_i)_{i=1}^n = 2 \sum_{i=0}^{n-1} (-1)^ix_{n-i} +(-1)^nt, \ \forall t\in \R.
\] 

We then form the stabilizing family $\{u_1, \cdots, u_{2n}\}$ of order $2n$ by setting 
\[
 u_{2i-1} = x_i = u_{2i}, i=1, ..., n.
\]
Hence, $\R$ admits an infinitely many stabilizing families of even orders.
\end{ex}

Similar construction as above can easily be done for the core quandle of any group. Precisely, we have the following.

\begin{pro}\label{pro:Core-2k}
Let $G$ be a non-trivial group. Then for all even natural number $2k\leq |G|$, there exists a non-trivial stabilizing family of order $2k$ for the \emph{core} quandle $Core(G)$. In particular, if $G$ is infinite, $Core(G)$ admits infinitely many stabilizing families.
\end{pro}

\begin{proof}
Recall that $Core(G)$ is $G$ as a set with the quandle structure defined by 
\[
g\rRack h = hg^{-1}h, \ \ g, h\in G.
\]   
Thus  $(g\rRack h_1)\rRack h_2 = h_2h_1^{-1}gh_1^{-1}h_2$,  and more generally

 \begin{eqnarray}\label{eq:takasaki}
g\rRack (h_i)_{i=1}^n= \prod_{i=0}^{n-1}h_{n-i}^{(-1)^i}g^{(-1)^n} \prod_{j=1}^{n}h_{j}^{(-1)^{n-j}}
\end{eqnarray}
 Then for any subset $\{x_1, \ldots ,x_n\} \subset X$ of length $n$, the family 
\[
 \{x_1,x_1, x_2, x_2,  \ldots , x_n, x_n\}
\]
 is a stabilizing family of order $2n$ for the \emph{core} quandle $Core(G)$.

\end{proof}

We moreover have the following complete characterization for groups whose core quandle have stabilizing families of odd order.

\begin{thm}\label{thm:Core-odd}
	Let $G$ be a group.  The core quandle of $G$ is $(2k+1)$--stable if and only if all elements of $G$ are $2$--torsions; in other words $G$ is isomorphic to $\bigoplus_{i\in I}\bb Z_2$, for a certain finite or infinite set $I$.
\end{thm}

\begin{proof}
	
	If $\{u_i \}_i \in \cal{S}^{2k+1}(Core(G))$, then is is easy to see that equation~\eqref{eq:takasaki} implies that the inversion map $g \mapsto g^{-1}$ in $G$ is an inner automorphism, therefore $G$ is abelian.  Now, since $G$ is an abelian group, then~\eqref{eq:takasaki} applied to the family $\{u_i\}_{i=1}^{2k+1}$ gives 
	\begin{eqnarray}\label{eq2:takasaki}
	g\rRack(u_i)_i = 2 \sum_{i=0}^{2k} (-1)^ih_{2k+1-i} - g = g, \ \forall g\in G;
	\end{eqnarray}
	 which implies in particular that 
	\[
	2\sum _{i=0}^{2k} (-1)^i u_{2k+1-i}=0,
	\]
	 therefore $2g=0, \forall g \in G$.  
	
Conversely, if all elements of $G$ are of order $2$, then it is obvious each element of $G$ is a stabilizer of $Core(G)$, hence $Core(G)$ is $1$--stable.
\end{proof} 

Assume $f:X\rTo Y$ is a rack homomorphism; {\em i.e.}, $f(x\rRack y)=f(x)\rRack f(y)$ for all $x,y\in X$. If $\{u_i\}_{i=1}^s\in \cal S(X)$, then $\{f(u_i)\}_{i=1}^s$ is a stabilizing family of Im$(f)$ which is a subrack of $Y$. In particular, if $f$ is onto and if $X$ is finitely stable, then $Y$ is finitely stable.

\section{Stable Alexander quandles}\label{SAQ}

In this section we are giving necessary and sufficient conditions for Alexander quandles to be stable, and provide a general algorithm to construct stabilizing families. More precisely, we are proving the following result.

\begin{thm}\label{thm:Alexander}
Let $\Gamma$ be a group, and let $M$ be a right $\bb Z[\Gamma]$--module. For each non-trivial $\g \in \Gamma$, define the (Alexander) quandle structure on $M$ given by 
\[
x\rRack^\g y :=(x-y)\cdot \g + y, \ \ x, y\in M.
\]
For $n\in \bb N$, let $F_\g:M^n\rTo M$ be the function defined by 
\[
F_\g(x_1,\ldots, x_n)=\sum_{i=1}^nx_i\cdot\g^{n-i}.
\] 
Then $(M,\rRack^\g)$ is $n$--stable if and only if $\g$ is of order $n$. Furthermore, if this condition is satisfied, $\cal S^n(M)$ is exactly the linear space of all solutions of the equation 
\begin{eqnarray}\label{eq:Alexander}
F_\g(x_1,\ldots, x_n)=0.
\end{eqnarray}
\end{thm}

Instead of proving directly the theorem, we are going to prove a more general result. But first let us give a few consequences.

\begin{cor}
Let $V$ be a complex vector space equipped with the Alexander quandle structure given by 
\[
x\rRack y = \a x + \b y, \ \ x, y\in V,
\]
where $\a\neq 0, \a\neq 1$, and $\b$ are fixed scalars such that $\a+\b = 1$. Then $V$ admits a stabilizing family $\{u_i\}_{i=1}^n$ of order $n$ if and only if the following hold:
\begin{itemize}
  \item[(i)] $\a$ is an $n^{th}$ root of unity;
  \item[(ii)] there exists an integer $0<l< n$, such that $\sum_{k=1}^{n}e^{-\frac{2i\pi kl}{n}}u_k = 0$.
\end{itemize}
\end{cor}

 \begin{proof}
 Consider the multiplicative group $\bb C^{\times}$ and think of the complex vector space $V$ as a $\bb Z[\bb C^{\times}]$--module. Then Theorem~\ref{thm:Alexander} applies.
 \end{proof}
\begin{cor}
Let $V$ be a vector space equipped with the quandle structure as above. Then, for all $n\in \bb N$ such that $\a$ is an $n^{th}$ root of unity, each $u\in V$ is an $n$--stabilizer. In other words, $V$ is an $n$--quandle..
\end{cor}

The following example shows that a quandle may have no stabilizing family at all.

\begin{ex}
Let $V$ be a vector space equipped with the quandle structure $$x\rRack y = (x+y)/2.$$ Then $\cal S(V)=\emptyset$
\end{ex}

\begin{df}
	Let $G$ be a group and $\vp$ be an automorphism of $G$, we define the \emph{$\vp$-conjugate} of $G$, denoted by $Conj_{\vp}(G)$, to the set $G$ with the quandle operation being
	\[
	g \rRack h:=h \vp(g)\vp(h^{-1}), \ \ \forall g, h \in G.
	\]	
\end{df}

In particular, if $\vp$ is the identity map of $G$, $Conj_{\Id}(G)$ is the usual \emph{conjugate quandle} $Conj(G)$; {\em i.e.}, the quandle operation is given by
\[
g\rRack h:= hgh^{-1}, \ \  h, g, \in G.
\]

\begin{pro}\label{pro:Conj_phi}
	Let $G$ and $\vp$ be as above. Then $Conj_{\vp}(G)$ is finitely stable if and only if there exists an integer $n$ such that $\vp ^n$ is an inner automorphism. Furthermore, $\{u_i\}_{i=1}^n\in \cal S^n(Conj_{\vp}(G))$ if and only if 
\begin{equation}\label{eq:Conj_phi}
  \vp^n = Ad_{u_n\vp(u_{n-1})\cdots \vp^{n-1}(u_1)}.
\end{equation} 
\end{pro}

\begin{proof}
Suppose that $Conj_{\vp}(G)$ is $n$--stable and $\{u_i\}_i\in \cal S^n(Conj_{\vp}(G))$. Then for all $g\in G$, we have
\[
\begin{array}{ll}
g\rRack (u_i) & =  \left[\prod_{i=0}^{n-1}\vp^i(u_{n-i}) \right]\vp^n(g)\left[\prod_{i=0}^{n-1}\vp^{n-i}(u_{i+1})^{-1}\right] \\
 & =  g
\end{array}
\]
so that $\vp^n(g) = Ad_{\prod_{i=0}^{n-1}\vp^i(u_{n-i})}(g)$; therefore $\vp^n$ is an inner automorphism. 
Conversely, if $\vp^n=Ad_u$ for some $u\in G$, then it is clear that the family $\{1, \ldots, 1, u\}$ is in $\cal S^n(Conj_{\vp}(G))$.
\end{proof}

As a consequence we have the following. 

\begin{cor}
For any group $G$, the conjugation quandle $Conj(G)$ is $n$--stable for all $n\leq 1$.
\end{cor}

We shall note that the abelian case is related to the torsion in the automorphism group as shown by the following. 

\begin{pro}
Let $G$ be a non-trivial abelian group and $\vp$ an automorphism of $G$. Define the quandle $G_{\vp}$ as the set $G$ together with the operation
\[
g\rRack h = \vp(g) + (\Id - \vp)(h), \ g, h \in G.
\]
Then the following are equivalent.
\begin{itemize}
  \item[(i)] $G_{\vp}$ is finitely stable.
  \item[(ii)] $\vp$ is a torsion element of $Aut(G)$, the automorphism group of $G$.
\end{itemize}
\end{pro}

\begin{proof}
For all finite subset $\{u_i\}_{i=1}^n$ of $G$, and every $g\in G$, we have 
\[
g\rRack (u_i)_{i=1}^n = \vp^n(g) + (\Id-\vp)\sum_{i=1}^n\vp^{n-i}(u_i).
\]
It follows that $\{u_i\}_{i=1}^n$ is a stabilizing family of order $n$ for $G_{\vp}$ if and only if $\sum_{i=1}^n\vp^{-i}(u_i) = 0$ and $\vp^n=\Id$.
\end{proof}

\section{The $n$-pivot of a group}\label{npivot}

Let $G$ be a group and $n$ a positive integer. The $n$--\emph{core} of $G$ was defined by Joyce in his thesis~\cite{Joyce:Thesis} as the subset of the Cartesian product $G^n$ consisting of all tuples $(x_1, \ldots, x_n)$ such that $x_1\cdots x_n = 1$. Moreover, the $n$-core of $G$ has a natural quandle structure defined by the formula 

\begin{eqnarray}\label{eq1:n-core}
(x_1, \ldots, x_n)\rRack (y_1, \ldots, y_n) := (y_n^{-1}x_ny_1, y_1^{-1}x_1y_2, \ldots, y_{n-1}^{-1}x_{n-1}y_n)
\end{eqnarray}

In this section we define $n$-\emph{pivot} of a group as a generalisation of the $n$-core and investigate its relationship with the center of racks. 

\begin{df}
Let $G$ be a group. For $n\in \bb N$, we define the \emph{$n$-pivot} of $G$ to be the subset of $G^n$ given by
\[
\cal P^n(G) = \{(x_1, \ldots, x_n)\in G^n \mid x_1\cdots x_n \in Z(G)\},
\]
where, as usual $Z(G)$ is the center of $G$.
\end{df}

It is straightforward to see that the formula~\eqref{eq1:n-core} defines a quandle structure on $\cal P^n(G)$.

 Moreover, we have the following result.

\begin{pro}\label{pro:core_n}
Let $G$ be a non-trivial group. Then for all $n\in \bb N$, there is a bijection $$\cal S^n(Cong(G)) \cong \cal P^n(G).$$
In particular, $\cal S^n(Conj(G))$ is naturally equipped with a quandle structure. Furthermore, if $G$ has a trivial center, then $\cal S^2(Cong(G)) = Core(G)$.
\end{pro}

\begin{proof}
We have already noticed that $Conj(G) = Conj_{\Id}(G)$. It follows from Proposition~\ref{pro:Conj_phi} that $\{u_i\}_{i=1}^n\in \cal S^n(Conj(G))$ if and only if 
\[
\Id = Ad_{u_n\cdots u_1},
\]
which is equivalent to $(u_1,\ldots, u_n)\in \cal P^n(G)$; hence the bijection $$\cal S^n(Conj(G)) \cong \cal P^n(G).$$ 
Next, if $G$ has trivial center of $G$, then for all  $n\geq 2$, $\cal S^n(Conj(G))$ coincides with the $n$--core of $G$. In particular we have $\cal S^2(Cong(G)) = Core(G)$.
\end{proof}


\section{Rack actions and approximate units} \label{rackaction}

In this section we give a more elaborated definition of rack actions and investigate some of their properties. Recall that (see for instance~\cite{Elhamdadi-Moutuou:Foundations_Topological_Racks}) a \emph{rack action} of a rack $X$ on a space $M$ consists of a map $M\times X \ni (m,x) \mto m\cdot x \in M$ such that
\begin{itemize}
\item[(i)] for all $x\in X$, the induced map $M\ni m \mto m\cdot x \in M$ is a bijection; and 
\item[(ii)] for all $m\in M, x,y\in X$ we have 
\begin{equation}\label{eq1:rack_action}
(m\cdot x)\cdot y = (m\cdot y)\cdot (x\rRack y).
\end{equation}
\end{itemize}
If $\{x_i\}_{i=1}^s$ is a family of elements in $X$, we will write $m\cdot(x_i)_i$ for 
\[
(\cdots (m\cdot x_1)\cdot \ \cdots )\cdot x_s.
\]

In this work, we require two additional axioms that generalize in an appropriate way the concept of group action. Precisely we give the following.

\begin{df}\label{df:rack_action}
An action of the rack $(X,\rRack)$ on the space $M$ consists of a map $M\times X \ni (m, x)\mto m\cdot x\in M$ satisfying equation~\eqref{eq1:rack_action} and such that
for all $\{u_1, \ldots, u_s\} \in \cal S(X)$, 
\begin{equation}
  m\cdot (u_i)_i = m\cdot (u_{\s(i)})_i, \ \ \forall m\in M,
\end{equation}
for all cycle $\s$ in the subgroup of $\fr S_n$ generated by the cycle $(2\  \ldots n \ 1) $

\end{df}

Here we should illustrate this definition; especially the following example gives the main motivation behind the first axiom.   

\begin{ex}
Let $G$ be a group acting (on the right) on a space $M$.  We then get a rack action of $Conj(G)$ on $M$ by setting $m \cdot g:=mg^{-1}$.  It is easy to check that for all $\{g_i\}_i \in \cal P^n(G)$, we have 
\[
m\cdot (g_i)_i= m \cdot (g_{\sigma(i)})_i,
\]
for all $m \in M$ and $\sigma$ in the subgroup generated by the cycle $(2\  \ldots n \ 1) $

\end{ex}

Note however that a rack action of $Conj(G)$ does not necessarily define a group action of $G$.

\begin{lem}\label{lem:action_stabilizing}
Let $X$ be a finitely stable rack acting on a set $M$. Then for any stabilizing family $\{u_i\}_{i=1}^s$, we have 
\[
(m\cdot x)\cdot (u_i)_i = (m\cdot (u_i)_i)\cdot x, 
\]
for all $m\in M, x\in X$.
\end{lem}

\begin{proof}
It is easy to verify that, thanks to~\eqref{eq1:rack_action}, for all finite subset $\{x_i\}_i \subset X$, we have
\begin{eqnarray}\label{eq:action_family}
(m\cdot x)\cdot (x_i)_i = (m\cdot (x_i)_i)\cdot (x\rRack (x_i)_i),
\end{eqnarray}
for all $m\in M, x\in X$. The result is therefore obvious if we take a stabilizing family $\{u_i\}_i$.
\end{proof}

\begin{df}
Let $X$ be a rack acting on an non-empty set $M$. For $\{x_i\}_{i=1}^s\in X$, we define 
\begin{itemize}
\item[(i)] the \emph{orbit} of $\{x_i\}_i$ as 
\[
M\cdot (x_i)_i = \{m\cdot (x_i)_{i=1}^s \mid m\in M\}.
\]
\item[(ii)] the \emph{fibre} of $M$ at $\{x_i\}_i$ to be 
\[
M_{\{x_i\}_i}=\{m\in M \mid m\cdot (x_i)_i = m\} \subset M\cdot x.
\]
\end{itemize}
\end{df}

\begin{df}
Let $X$ and $M$ be as above. For $m\in M$, we define the \emph{stabilizer} of $m$ to be 
\[
X[m] = \{x\in X \mid m\cdot x = m\}
\]
\end{df}

The following is straightforward. 

\begin{lem}
Let $X$ and $M$ be as above. Then, all stabilizers are subracks of $X$.
\end{lem}

\begin{df}
A rack action of $X$ on an nonempty set $M$ is {\em faithful} if for each $m\in M$, the map $X\ni x\mto m\cdot x\in M$ is one-to-one.
\end{df}

\begin{lem}\label{lem:faithful_action}
Let $X$ be a nonempty set. A rack structure $\rRack$ on $X$ is trivial if and only if $(X,\rRack)$ acts faithfully on a nonempty set $M$ and the action satisfies $(m\cdot x)\cdot y = (m\cdot y)\cdot x$ for all $x,y\in X, m\in M$.
\end{lem}

\begin{proof}
If $X$ acts faithfully on a nonempty set $M$ and $(m\cdot x)\cdot y= (m\cdot y)\cdot x$, then
\[
(m\cdot x)\cdot y = (m\cdot y)\cdot x = (m\cdot x)\cdot (y\rRack x ), \ \forall x,y\in X, m \in M;
\]
therefore, by faithfulness, $y\rRack x=y$ for all $x,y\in X$, hence the rack structure is trivial.
Conversely, suppose the rack structure on $X$ is trivial $X$. Then one defines a faithful rack action of $X$ on the set $\bb CX$ of all complex-valued continuous functions $\a:X\rTo \bb C$ by setting
\[
\a\cdot x:= \a +\d_x, \ \a\in \bb C, x\in X.
\]
where, as usual, $\d_x$ is the characteristic function of $x$.
\end{proof}

\begin{df}
Let $X$ be a rack acting on a non-empty set $M$. 
\begin{itemize}
\item[(i)] An {\em approximate unit} for the rack action is a finite subset $\{t_i\}_{i=1}^r\subset X$ such that $m\cdot (t_i)_i=m$ for all $m\in M$.
\item[(ii)] An element $t\in X$ is an $r$--{\em unit} for the rack action if the family $\{t\}_{i=1}^r$ is an approximate unit for the rack action.
\item[(iii)] The rack action of $X$ on $M$ is called {\em $r$--periodic} if each $t\in X$ is an $r$--unit for the rack action.
\item[(iv)] A rack action is said to be {\em strong} if every stabilizing family of the rack is an approximate unit.
\end{itemize}

\end{df}


\begin{ex}
Let $(X, \rRack)$ be a rack. Then we get a rack action of $X$ on its underlying set by defining $m\cdot x := m\rRack x, \ m, x\in X$. It is immediate that a finite subset $\{t_i\}_i\subset X$ is an approximate unit for this rack action if and only if it is a stabilizing family for the rack structure, hence it is a strong action. Furthermore, this action is $r$--periodic if and only if $X$ is an $r$--rack in the sense of Joyce~\cite{Joyce:Thesis}.
\end{ex}

\begin{lem}\label{lem:strong_action}
Let $X$ be a finite rack acting strongly on the set $M$. Let $x\in X$ be of order $k$ in $X$. Then, every element $m\cdot (x_i)_{i=1}^r$, where $x$ appears $k$ times in the sequence $x_1,\ldots, x_r$, can be written as $m\cdot (y_j)_{j=1}^{r-k}$. 
\end{lem}

\begin{proof}
Since $X$ is finite, every element $x\in X$ has a finite order; that is there exists $k$ such that $x$ is a $k$--stabilizer for $X$, hence since the action is strong, $m\cdot ^kx=m$ for all $m\in M$. Now the result follows by applying~\eqref{eq1:rack_action} to $m\cdot (x_i)_{i=1}^r$ repeatedly until the $x$'s appear next to each other.
\end{proof}

\section{Twisted actions and rack cross-products}\label{twist}

In this section we examine the case of rack actions on racks.  This generalizes the construction of extensions of racks using the notion of dynamical cocycles \cite{And-Grana}. Specifically, let $X$ and $Q$ be racks whose rack structures are indistinguishably denoted by $\rRack$. A rack action $Q\times X \ni (p,x)\mto p\cdot x \in Q$ is called an {\em $X$-action by rack automorphisms} on $Q$ if for each $x\in X$, the induced bijection 
\[
Q\ni p\mto p\cdot x\in Q
\]  
is a rack automorphism of $Q$. 

Throughout, given two sets $A$ and $B$, $\cal F(A,B)$ will denote the set of all maps from $A$ to $B$. 

\begin{df}\label{def:X-cocycle}
Let $X$ be a rack acting by automorphisms on the rack $Q$. An $X$--{\em cocycle} on $Q$ is a map 
\[
\begin{array}{lccc}
\partial :& X\times X & \rTo & \cal F(Q\times Q, Q)\\
 & (x,y) & \mto & \partial_{x,y}
\end{array}
\]
 such that
 \begin{itemize}
 \item[(i)] for each $x,y\in X$ and each $q\in Q$, the induced map $$\partial_{x,y}(-, q):Q \ni p\mto \partial_{x,y}(p,q)\in Q$$ is a rack automorphism;
 \item[(ii)] (Equivariance) for all $x,y\in X$ and $p,q\in Q$, we have 
 \begin{eqnarray}\label{eq1:X-cocycle}
 \partial_{x,y}(p\cdot t, q)=\partial_{x,y}(p,q)\cdot (t\rRack y),
 \end{eqnarray}
 for all $t\in X$;
 \item[(iii)](Cocycle condition) for all $x,y,z\in X$ and all $p,q,r\in Q$, the following relation holds
 \begin{eqnarray}\label{eq2:X-cocycle}
 \partial_{x\rRack y, z}(\partial_{x,y}(p,q),r) = \partial_{x\rRack z, y\rRack z}(\partial_{x,z}(p,r),\partial_{y,z}(q,r)).
 \end{eqnarray} 
 \end{itemize}
\end{df}

\begin{df}
Let $X$ be a rack. A {\em twisted $X$--action} on a rack $Q$ is a rack action of $X$ by automorphisms on $Q$ together with an $X$--cocycle on $Q$. The triple $(Q,X,\partial)$ will be called a {\em twisted rack dynamical system}.
\end{df}

We shall notice that the idea of {\em extension by dynamical cocycle} introduced by Andruskiewitsch and Gra\~{n}a in~\cite{And-Grana} is a special case of a twisted dynamical system in our sense. Indeed, when $Q$ is endowed with the trivial rack structure ({\em i.e.}, $p\rRack q=p$), and the trivial $X$--action, then relation~\eqref{eq1:X-cocycle} is trivially satisfied. 

\begin{ex}\label{eq:twisted_cocycle}
If $X$ acts by rack automorphism on the rack $Q$, then  define the map $\partial:X\times X\rTo \cal F(Q\times Q, Q)$ by 
\[
\partial_{x,y}(p,q):=p\cdot y, \ {\rm for \ } p,q\in Q.
\]
It is easy to check that $(Q,X,\partial)$ is a twisted rack dynamical system.
\end{ex}

\begin{df}
Let $(Q,X,\partial)$ be a twisted rack dynamical system. The {\em rack cross-product} $Q\rtimes_\partial X$ of $Q$ by $X$ is the Cartesian product $Q\times X$ together with the binary operation
\begin{eqnarray}\label{eq:semi_direct}
(p,x)\rRack_\partial (q,y) = (\partial_{x,y}(p,q),x\rRack y),
\end{eqnarray}
for $(p,x), (q,y)\in Q\times X$.
\end{df}

\begin{lem}
Let $X$ and $Q$ be as above. Then, together with the operation~\eqref{eq:semi_direct}, $Q\rtimes_\partial X$ is a rack.
\end{lem}

\begin{proof}
Left to the reader.
\end{proof}

One can easily verify that given a twisted rack dynamical system $(Q,X,\partial)$, if the cross-product $Q\rtimes_\partial X$ is finitely stable, then so is $X$, but $Q$ may not be so. Precisely, we have the following

\begin{pro}
Let $(Q,X,\partial)$ be a twisted rack dynamical system. Then, the rack cross-product $Q\rtimes_\partial X$ is finitely stable if and only if there exists a finite family $(\xi_i, t_i)_{i=1}^n$ of elements in $Q\times X$ such that $\{t_i\}_i\in \cal S(X)$ and such that the function
\[
p\mto \partial_{x\rRack (t_i)_{i=1}^{n-1},t_n}(\partial_{x\rRack(t_i)_{i=1}^{n-2},t_{n-1}}(\ldots (\partial_{x\rRack t_1,t_2}(\partial_{x,t_1}(p,\xi_1),\xi_2))\ldots),\xi_n), 
\]
is the identity function on $Q$, for all $x\in X$.
\end{pro}

\begin{ex}
Let $(\cal H,\pi)$ be a representation of a finite group $G$ with non-trivial center. Denote $\xi g:=\pi_g(\xi)$ for $\xi\in \cal H, g\in G$. Equip $\cal H$ with the quandle structure $\xi\rRack \eta:=\frac{\xi + \eta}{2}$. Then $Conj(G)$ acts by rack automorphisms on $\cal H$ under the operation $\xi\cdot g:=\xi g^{-1}$. 

Fix a non-zero vector $v_0\in \cal H$, and let $\zeta_0:=\frac{1}{\#G}\sum_{g\in G}v_0g$. Then $\zeta_0\cdot h = \zeta_0$ for all $h\in G$. Let $t\neq 1$ be a solution of $x^3-1=0$ in $\bb C$, and for $g,h\in G$, let $\partial_{g,h}:\cal H\times \cal H\rTo \cal H$ be the function given by
\[
\partial_{g,h}(\xi,\eta) :=t\cdot \xi h^{-1} + (1-t)\zeta_0, \ \xi,\eta\in \cal H.
\]
One can verify that the triple $(\cal H, Conj(G), \partial)$ satisfies all the axioms of a twisted rack dynamical system. Furthermore, $\cal H\rtimes_\partial Cong(G)$ is finitely stable; although $\cal H$ is not finitely stable, thanks to Theorem~\ref{thm:Alexander}. To see this, let $g_1, g_2\in Z(G)$ and $g_3=(g_1g_2)^{-1}$, so that $(g_1,g_2,g_3)\in \cal S^3(Conj(G))$. Let $\eta_1,\eta_2, \eta_3$ in $\cal H$. Then for all $g\in G$, we get 

\[
\begin{array}{lcl}
\partial_{(g\rRack g_1)\rRack g_2,g_3}(\partial_{g\rRack g_1,g_2}(\partial_{g,g_1}(\xi,\eta_1),\eta_2),\eta_3) & = & \xi + (1-t)(t^2+t+1)\zeta_0 \\
& = & \xi 
\end{array}
\]
for all $\xi \in \cal H$; which implies that $(\eta_i,g_i)_i\in \cal S^3(\cal H\rtimes_\partial Conj(G))$.

\end{ex}

As a consequence of the above proposition, we have

\begin{cor}
Let $X$ be a rack acting by rack automorphisms on the rack $Q$. Let $Q$ be equipped with the $X$-cocycle $\partial$ of Example~\ref{eq:twisted_cocycle} ({\em i.e.} $\partial_{x,y}(p,q)=p \cdot y$). Then $Q\rtimes_\partial X$ is finitely stable if and only if there exists $\{t_i\}_i\in \cal S(X)$ which is an approximate unit for the rack action of $X$ on $Q$.
\end{cor}

\begin{ex}
Recall from~\cite{Elhamdadi-Moutuou:Foundations_Topological_Racks} that a left ideal in a rack $X$ is subrack $T$ of $X$ such that $T\rRack X\subseteq T$. Let $T$ be such a left ideal. Then the operation $T\times X\ni (t,x)\mto t\rRack x\in T$ defines an action of $X$ by rack automorphisms on $T$. Let $\partial_{x,y}(t,s):=t\rRack y$. Then the cross-product $T\rtimes_\partial X$ is finitely stable if and only if $X$ is. 
\end{ex}


\section{$X$--bundle of racks}\label{Xbundle}

In this section we examine the case of families of rack structures on a set parametrized by a given rack. More specifically, we introduce bundles of racks which are generalization of the notion of $G$--{\em families of quandles} as defined in~\cites{Ishii:G-family_Quandles, Nosaka:Quandles_Cocycles}.

\begin{df}
Let $X$ be a rack. An $X$--{\em bundle of racks} $(\cal A,\star_x^y)_{x,y\in X}$ consists of a set $\cal A$, and a family of racks structures $\star_x^y$ on $\cal A$ such that 
\begin{eqnarray}\label{eq:X-bundle}
(a\star_x^yb)\star_{x\rRack y}^z c = (a\star_x^zc)\star_{x\rRack z}^{y\rRack z} (b\star_y^z c),
\end{eqnarray}
for all $a,b,c\in \cal A$ and all $x,y,z\in X$.
\end{df}

\begin{lem}
Let $(Q,X,\partial)$ be a twisted rack dynamical system. Then the family $(Q,\partial_{x,y}(-, -))_{x,y\in X}$ is an $X$--{\em bundle of racks} if and only if for all $x,y\in X$ we have 
\begin{eqnarray}\label{eq:Cocycle-vs-X-bundle}
\partial_{x,y}(\partial_{x,y}(p,q),r) = \partial_{x,y}(\partial_{x,y}(p,r), \partial_{x,y}(q,r)), \forall p,q,r\in Q.
\end{eqnarray} 
\end{lem}

\begin{proof}
The identity~\eqref{eq:Cocycle-vs-X-bundle} means that the operation $p\star_x^yq:=\partial_{x,y}(p,q)$ defined on $Q$ is distributive. Moreover, the invertibility of the right multiplication $$-\star_x^y q: Q\ni p \mto p\star_x^y q\in Q$$ is automatic from Definition~\ref{def:X-cocycle}. Also, notice that equation~\eqref{eq:X-bundle} is satisfied through~\eqref{eq2:X-cocycle}. 
\end{proof}

\begin{df}
Let $X$ and $Y$ be racks, and $(\cal A, \star_x^y)_{x,y\in X}$ be an $X$--bundle of racks. For any rack morphism $f:Y\rTo X$, we define the {\em pull-back} of $(\cal A, \star_x^y)_{x,y\in X}$, denoted by $(f^\ast \cal A, \star_u^v)_{u,v\in Y}$ as the set $\cal A$ equipped with the family of binary operations $\star_u^v, u,v\in Y$, 
\[
a\star_u^vb:=a\star_{f(u)}^{f(v)}b, \ \ u, v\in Y, a,b\in \cal A.
\]
\end{df}

\begin{lem}
Let $X$, $Y$, $(\cal A, \star_x^y)_{x,y\in X}$, and $f:X\rTo Y$ be as above. Then the pull-back $(f^\ast A, \star_u^v)_{u,v\in Y}$ is a $Y$--bundle of racks. 
\end{lem}

As we already mentioned in the beginning of the section, $G$--families of quandles are special cases of our bundles of racks. Specifically we have the following example.

\begin{ex}
Recall from~\cite{Ishii:G-family_Quandles} that given a group $G$, a $G$--{\em family of quandles} is a non-empty set $X$ together with a family of quandle structures $\rRack^g$, indexed by $G$, on $X$, satisfying the following equations:
\begin{itemize}
\item $x\rRack^{gh}y =  (x\rRack^gy)\rRack^hy$, and $x\rRack^ey=y$, for all $x,y\in X,  \ g,h\in G$;
\item $(x\rRack ^gy)\rRack^hz = (x\rRack^hz)\rRack^{h^{-1}gh}(y\rRack^hz)$, for all $x,y,z\in X, g,h\in G$.
\end{itemize} 

Now, given such a $G$--family of quandles, we construct the $Conj(G)$--bundle of racks $(X,\star_g^h)_{g,h\in G}$ by letting 
\[
x\star_g^hy:= x\rRack^hy, x,y\in X, g,h\in G.
\]
\end{ex}

We shall also note that our definition of bundles of racks generalizes the notion of $Q$--family of quandles (where $Q$ is a quandle) defined in~\cite[p.819]{Ishii:G-family_Quandles}. Indeed, by using the same construction as in the above example, one can easily show that any $Q$--family of quandles defines a $Q$--bundle of racks (quandles).
\section{Rack representations}\label{rackrep}

In this section we introduce first elements of representation theory for abstract racks and quandles. 

\begin{df}
A \emph{representation} of a rack $X$ is a vector space $V$ equipped with a rack action of $X$ such that for all $x \in X$, the induced map $V \ni v \mapsto v \cdot x \in V$ is an isomorphism of vector spaces. Equivalently, a representation of $X$ consists of a vector space $V$ and a rack homomorphism 
\[
\pi: X\rTo Conj(GL(V));
\]
{\em i.e.}, $\pi_{x\rRack y} = \pi_y\pi_x\pi_y^{-1}, \ \forall x, y\in X$. The map $\pi$ will be denoted as $\pi^V$ where there likely to be a confusion.
\end{df}

\begin{ex}
Every representation $(V,\rho)$ of a group $G$ naturally defines representation of the quandle $Conj(G)$. 
\end{ex}

Let us give an example of rack representation analogous to the regular representation in group theory. 

\begin{ex}\label{ex:regular}
Let $X$ be a finite rack and let $\bb C X$ be the vector space of complex valued functions on $X$, seen as the space of formal sums $f = \sum_{x\in X} a_xx$, where $a_x\in \bb C, x\in X$.

We construct the {\em regular representation} of the rack $X$ 
\[
\lambda : X\rTo Conj(GL(\bb CX))
\] 
by $\lambda_t(f)(x):=f(R_t^{-1}(x))$, where as usual $R_t$ is the right "translation" $X\ni y\rTo y\rRack t\in X$.  
\end{ex}

\begin{df}
Let $V$ and $W$ be representations of the rack $X$. A linear map $\phi: V\rTo W$ is called $X$-\emph{linear} if for each $x\in X$ the following diagram commutes
\[
\xymatrix{
V \ar[r]^{\pi^V_x} \ar[d]_{\phi} & V \ar[d]^{\phi} \\
W \ar[r]^{\pi^W_x} & W
}
\]
If in addition $\phi$ is an isomorphism, the two representations $V$ and $W$ are said to be \emph{equivalent}, in which case $\phi$ is called an {\em equivalence of rack representations}.
\end{df}

If $V$ is a representation of the rack $X$, a subspace $W\subset V$ is called a {\em subrepresentation} if for all finite family $\{x_i\}_i\subseteq X$, $W\cdot (x_i)_i \subseteq W$. We have the following immediate observation.

\begin{lem}\label{lem:subrep_Ker_Im}
Let $V$ and $W$ be representations of the rack $X$ and let $\phi:V\rTo W$ be a linear map. If $\phi$ is $X$-linear, then $\ker \phi$ and ${\rm Im\ } \phi$ are subrepresentations of $V$ and of $W$, respectively. 
\end{lem}

\begin{df}
A rack representation $V$ of $X$ is said to be {\em irreducible} if it has no proper subrepresentation; {\em i.e.}, if $W$ is a subrepresention of $V$, then either $W=\{0\}$ or $W=V$.
\end{df}

\section{Strong Representations}\label{strong_rep}

\begin{df}
A {\em strong representation} of $X$ is a representation $V$ such that the rack action is strong.

We denote by $Rep_s(X)$ the set of equivalence classes of strong irreducible finite dimensional representations of $X$. 
\end{df}

One can check that as in the group case, $Rep_s(X)$ is an abelian group under tensor product of strong irreducible representations. 

\begin{rmk}
In view of Proposition~\ref{pro:core_n}, we notice that if $G$ is an abelian group, the only strong representation of $Conj(G)$ is the trivial one.
\end{rmk}

\begin{ex}
The regular representation $ (\bb{C}X, \lambda)$ of a rack $X$ defined in Example \ref{ex:regular} is clearly strong.
\end{ex}

\begin{ex}
Let $(\bb{Z}_3, \rRack)$, with $x \rRack y=2y-x$, be the dihedral quandle. Define $\rho: \bb{Z}_3 \rTo Conj(GL(\bb{C}^3))$  as the rack representation induced by the reflections on $\bb C^3=Span\{e_1,e_2,e_3\}$
\[
\rho_0= (2 \; 3), \; \rho_1=(1 \; 3),\; \text{and}\; \rho_2=(1 \; 2).
\]
Then $\rho$ is a strong representation of $\bb{Z}_3$ as a quandle, although it is clearly not a group representation. Note, however, that this is a reducible representation; indeed, the one-dimensional subspace spanned by $(1,1,1)$ is a subrepresentation.
\end{ex}

\begin{ex}\label{ex:involution_rep}
Suppose $X$ is an involutive rack, that is, $(x\rRack y)\rRack y=x$ for all $x,y\in X$; which means each $y\in X$ is a $2$--stabilizer for $X$. Then every pair $(V,\tau)$, where $V$ is a vector space and $\tau:V\rTo V$ is a linear involution, gives rise to a strong representation $\tilde{\tau}:X\rTo Conj(GL(V))$ by setting $\tilde{\tau}_x(v)=\tau(v)$ for all $x\in X, v\in V$.
\end{ex}

 In fact we have the following.
\begin{lem}\label{lem:inv_rep}
Let $X$ be an involutive rack. The assignment $(V,\tau)\mto \tilde{\tau}$ defines a covariant functor from the category $\cal V_{inv}$, whose objects are pairs $(V,\tau)$ of vector spaces equipped with involutions and whose morphisms are vector space morphisms intertwining the involutions, to the category of strong representations of $X$.
\end{lem}

\begin{pro}
Suppose the rack $X$ is finite, involutive, and connected; that is, $X$ has only one orbit. Then, the regular representation of $X$ corresponds to a conjugacy class of the symmetric group $\frak S_n$ where $n={\#X}$. More generally, if $X$ has $k$ connected components, then the regular representation corresponds to $k$ conjugacy classes in $\frak S_n$.  
\end{pro}

\begin{proof}
Since the representation $\lambda\colon X\rTo Conj(GL(\bb CX))$ is strong and $X$ is connected, the involutions $\lambda_t:\bb CX\rTo \bb CX$ are all equivalent, hence they define a partition of the dimension $n$ of $\bb CX$. The result then follows from the correspondence between partitions of $n$ and conjugacy classes of the symmetric group. 
\end{proof}

\begin{thm}\label{thm:irr-stable}
Let $X$ be a finite rack. Then every irreducible strong representation of $X$ is either trivial or finite dimensional.
\end{thm}

\begin{proof}
Let $X$ be a finite rack of cardinality $n$ and let $(V,\pi)$ be a nontrivial irreducible strong representation of $X$. Then, since every element $x\in X$ is a $k$--stabilizer for $X$ where $k$ is a factor of $n!$, we have $\pi_x^{n!}=\Id_V$ for each $x\in X$. Fix a non-zero vector $v\in V$ and let $E_v$ be the subspace
\[
{\rm Span}\left\{v\cdot (x_i)_{i=0}^s\mid s=0,\ldots, (n+1)!, \{x_i\}_{i=1}^s\subseteq X\right\}
\]
of $V$, where we have used the convention that $v\cdot \emptyset =v$. Then $E_v$ is a finite dimensional complex vector space. Furthermore, $E_v$ is invariant under the rack $X$--action, thanks to Lemma~\ref{lem:strong_action} and to the fact that any sequence of $(n+1)!$ elements in $X$ has at least one element repeated at least $n!$ times. This means that $E_v$ is a subrepresentation of $V$. Therefore, since $V$ is irreducible and $E_v\neq \{0\}$, we have $V=E_v$.
\end{proof}

We shall observe that the case of involutive racks the above result becomes more precise with regard to the dimension of the irreducible representations. 

\begin{lem}\label{lem:involutive_rack}
Suppose $X$ is an involutive rack (finite or infinite). Every irreducible strong representation of the form $(V,\tilde{\tau})$ coming from the category $\cal V_{inv}$ is one-dimensional.
\end{lem}

\begin{proof}
Let $V_+:=\{v\in V\mid \tau(v)=v\}$. If $V_+$ is a nonzero space, and $v_0$ is a nonzero vector in $\in V_+$, then the subspace $Span\{v_0\}$ is clearly a subrepresentation of $V$, hence $V$ is one-dimensional. 
On the other hand, if $V_+=\{0\}$, then $\tau(v)=-v$ for all $v\in V$. Hence each nonzero vector $v \in V$ spans a subrepresention of $V$. Therefore, since $V$ is irreducible, it spans $V$.
\end{proof}

\begin{cor}\label{cor1:involutive_rack}
Let $X$ be a connected involutive rack. If $(V,\pi)$ is a strong irreducible representation satisfying the property that there exists an $x_0\in X$ such that $V^{x_0}_+=\{v\in V\mid \pi_{x_0}(v)=v\}=0$, then $V$ is one-dimensional. 
\end{cor}

\begin{proof}
We write $V=V^{x_0}_+\oplus V^{x_0}_-$ where the second summand is the subspace of $V$ consisting of vectors $v$ such that $\pi_{x_0}(v)=-v$. The assumption amounts to saying $\pi_{x_0}=-\id_V$. Now, since $X$ is connected, every $z\in X$ can be written as $z=x_0\rRack y$ for some $y\in X$. Hence $\pi_z=\pi_y\pi_{x_0}\pi_y$, which means $\pi_z$ and $\pi_{x_0}$ are conjugate involutions on $V$. Therefore the representation$(V,\pi)$ comes from a fixed involution as in the lemma.
\end{proof}

\begin{thm}\label{thm:strong_rep_connected}
Every irreducible strong representation of an involutive connected finite rack $X$ is one-dimensional.
\end{thm}

\begin{proof}
Write $X=\{x_1,\ldots, x_n\}$. Let $(V,\pi)$ be a strong irreducible representation of $X$, and $\{e_1,\cdots, e_m\}$ be a basis for $V$. As in the proof of Corollary~\ref{cor1:involutive_rack}, since $\pi_x:V\rTo V$ is an involution for each $x\in X$, there are subspaces $V^{x_i}_+, i=1,\ldots, n$ of $V$ with the property that for every $i=1,\ldots,n$, $\pi_{x_i}(v)=v, \forall v\in V^{x_i}_+$; specifically, $V^{x_i}_+$ is the eigenspace of $\pi_{x_i}$ associated to the eigenvalue $1$. If one of the $V^{x_i}_+$'s is trivial, then so are all of the others, and $V$ is one-dimensional, thanks to Corollary~\ref{cor1:involutive_rack}. 
Suppose then that $V^{x_i}_+\neq \{0\}$ for all $i=1,\ldots, n$. Let $V_+:=\sum_{i=1}^nV^{x_i}_+$ be the subspace of $V$ spanned by $\bigcup_iV^{x_i}_+$. Then $V_+$ is invariant under the rack action by $X$. Indeed, for all $t\in X$, the restriction map $\pi_t:V^{x_i}_+\rTo V^{x_i\rRack t}_+$ is an isomorphism of vector spaces. It follows that $V_+$ is a subrepresentation of $V$. Therefore, since $V$ is irreducible and $V_+\neq \{0\}$, we have $V=V_+$. 

We thus have that $\cap_{k=1}^nV^{x_k}_+ \neq \{0\}$. Indeed, since $X$ is connected, we may fix $t_1, \ldots, t_{n-1}\in X$ such that $x_{i+1}=x_i\rRack t_i,\ i=1, \ldots, n-1$. Let $\vp\colon V\rTo V$ be the morphism of vector spaces given by 
\[
\vp:=\Id+\pi_{t_1}+\pi_{t_2}\pi_{t_1}+ \cdots + \pi_{t_{k-1}}\cdots \pi_{t_1} +\cdots + \pi_{t_{n-1}}\cdots \pi_{t_1}
\] 
Then for $k=1, \ldots, n$, we have 
\[
\vp(V^{x_k}_+) = V^{x_k}_+ + \cdots + V^{x_1}_+ + \cdots + V^{((x_1\rRack t_2)\rRack \cdots )\rRack t_{n-1}}_+
\]
which means that $V^{x_1}_+$ is in the intersection of all the spaces $\vp(V^{x_k}_+)$; hence $\cap_k\vp(V^{x_k}_+) \neq \{0\}$.

Now, any non zero vector $v_0\in \cap_kV^{x_k}_+$ will be invariant under $\pi$, therefore it will span $V$.
\end{proof}

\begin{cor}
If $G$ is a finite group, then every nontrivial irreducible strong representation of $Core(G)$ has one dimension.
\end{cor}


\section{Rack characters and duality}\label{duality}

In this section we give an analogue construction of the {\em Pontryagin dual} for racks and quandles. 

\begin{df}
Let $X$ be a rack. A {\em character} of $X$ is a map $\phi$ on $X$ with values on the unitary group $U(1)$ such that $\phi(x\rRack y) = \phi(x)$; in other words, $\phi$ is function from $X$ to $U(1)$ that is constant on the orbits of $X$. 
 The set of all characters of $X$, called the {\em  Pontryagin dual} of $X$, will be denoted by $D_qX$.
\end{df} 

\begin{ex}
As in the finite group case, any representation of a finite rack gives rise to a character of the rack. Indeed, suppose $(V,\pi)$ is a representation of a finite rack $X$. One defines $\phi_\pi:X\rTo U(1)$ by $\phi_\pi(x):=Tr(\pi_x)$. Then for all $x,y\in X$, we have $\phi_\pi(x\rRack y)=Tr(\pi_y\pi_x\pi_y^{-1})=Tr(x)=\phi_\pi(x)$.
\end{ex}

We shall observe the following.

\begin{lem}
For any rack $X$, $D_qX$ is an Abelian group under point-wise product, with inverse given by the point-wise conjugate. Moreover
\[
D_qX\cong \bigoplus_{\# orb(x)}U(1),
\]
where for $x\in X$, $orb(x)$ is the set of all $z\in X$ such that $z=x\rRack y$ for some $y\in X$.
\end{lem}

\begin{thm}
If $X$ is a finite involutive and connected rack, then there is an isomorphism of Abelian groups $Rep_s(X)\cong D_qX$.
\end{thm}

\begin{proof}
Since $X$ is finite, involutive, and connected, then by Theorem~\ref{thm:strong_rep_connected}, a strong irreducible representation of $X$ is but a character of $X$. Now, the map $(V,\pi) \mto \phi_\pi$ from the collection of all strong irreducible representations of $X$ to $D_qX$ satisfies $\phi_{\pi_1\otimes \pi_1}=\phi_{\pi_1}\phi_{\pi_2}$ and induces the desired isomorphism of Abelian groups.
\end{proof}

\begin{ex}
Since $\bb Z$ is abelian, the rack structure of $Conj(\bb Z)$ is trivial, $D_qConj(\bb Z)$ is the set of all sequences on $U(1)$.  
\end{ex}

\begin{ex}
Let $\bb Z_4$ be the dihedral quandle with two components $\{0,2\}$ and $\{1,3\}$. Then $D_q\bb Z_4 \cong U(1)\oplus U(1)$.
\end{ex}

\begin{ex}
Let $G$ be an Abelian group. Then, since $Core(G)$ is a connected quandle, we have $$D_qCore(G[{\frac 12}]) \cong U(1).$$
\end{ex}

More generally, we have the following for Alexander quandles. 

\begin{ex}
Let $M$ be a $\bb Z[t,t^{-1}]$--module equipped with the Alexandle quandle structure $x\rRack^ty:=(x-y)t + y$. Then $D_qM[\frac{1}{1-t}]\cong U(1)$. In particular, if $E$ is a complex or real vector space equipped with the Alexander quandle operation $\xi \rRack \eta = {\frac 12}(\xi+\eta)$, then $D_qE=U(1)$.
\end{ex}

\begin{pro}
Let $G$ be an Abelian group generated by a finite set of cardinality $n$. Then 
\[
D_qCore(G) \cong U(1)^{2^n}.
\]
\end{pro}

\begin{proof}
Indeed, for all $k\in \bb Z$ and all $x\in G$, we have $$\phi((2k+1)x) = \phi(-x\rRack kx) = \phi(-x), \ {\rm and \ } \phi(-x)=\phi(x\rRack 0)=\phi(x).$$ Hence, $\phi((2k+1)x) = \phi((2k-1)x) = \phi(x)$ for all $k\in \bb Z$ and all $x\in G$. Furthermore, $\phi(2kx) = \phi(0\rRack kx) = \phi(0), \forall k\in \bb Z$. It follows that if $G = <S>\cup <-S>$ and $S=\{s_1, \ldots, s_n\}$, then one can check that every character $\phi \in D_qCore(G)$ is completely determined by its $2^n$ values 
\begin{itemize}
\item $\phi(0), \phi(s_1), \ldots, \phi(s_n)$;
\item $ \phi(s_1+s_2), \ldots, \phi(s_1+s_n), \ldots, \phi(s_{n-1}+s_n)$;
\item  $\phi(s_1+s_2+s_3), \ldots, \phi(s_1 + s_2 + s_n)$;
\item ...
\item $\phi(s_1+ s_2 + \ldots +s_n)$.
\end{itemize}
in $U(1)$. This gives a bijection between $D_qCore(G)$ and the set of all maps from the power set of $S$ to $U(1)$.
\end{proof}

\begin{ex}\label{pro:dual_Z_ce}
Let $\bb Z_{ce}$ be $\bb Z$ equipped with the core structure $m\rRack n = 2n-m, \ m,n\in \bb Z$.  Then $D_q\bb Z_{ce} = U(1)\oplus U(1)$. In particular, every character $\phi$ of $\bb Z_{ce}$ is completely determined by its two values $\phi(0)$ and $\phi(1)$. 
\end{ex}


\end{document}